\documentclass[12pt]{amsart}
\usepackage{fullpage, amsmath, amssymb, amsfonts, amsthm, stmaryrd, url, hyperref, combelow}
\usepackage[all]{xy}

\newtheorem{theorem}{Theorem}[section]
\newtheorem{lemma}[theorem]{Lemma}
\newtheorem{cor}[theorem]{Corollary}

\newtheorem{prop}[theorem]{Proposition}
\theoremstyle{definition}
\newtheorem{defn}[theorem]{Definition}
\newtheorem{hypothesis}[theorem]{Hypothesis}

\newtheorem{remark}[theorem]{Remark}
\newtheorem{convention}[theorem]{Convention}

\numberwithin{equation}{theorem}

\newcommand{\FF}{\mathbb{F}}

\newcommand{\Qp}{\mathbb{Q}_p}
\newcommand{\QQ}{\mathbb{Q}}
\newcommand{\RR}{\mathbb{R}}

\newcommand{\ZZ}{\mathbb{Z}}

\newcommand{\calH}{\mathcal{H}}

\newcommand{\calM}{\mathcal{M}}

\newcommand{\calR}{\mathcal{R}}

\newcommand{\gothm}{\mathfrak{m}}
\newcommand{\gotho}{\mathfrak{o}}

\DeclareMathOperator{\bd}{bd}
\DeclareMathOperator{\Frac}{Frac}

\DeclareMathOperator{\inte}{int}
\DeclareMathOperator{\Maxspec}{Maxspec}

\DeclareMathOperator{\Spa}{Spa}

\begin{document}

\title{Noetherian properties of Fargues-Fontaine curves}
\author{Kiran S. Kedlaya}
\address{Department of Mathematics, University of California, San Diego, La Jolla, CA 92093, USA}
\email{kedlaya@ucsd.edu}
\date{July 6, 2015}

\begin{abstract}
We establish that the extended Robba rings associated to a perfect nonarchimedean field of characteristic $p$, which arise in $p$-adic Hodge theory
as certain completed localizations of the ring of Witt vectors, are strongly noetherian Banach rings; that is, the completed polynomial ring in any number of variables over such a Banach ring is noetherian. This enables Huber's theory of adic spaces to be applied to such rings. We also establish that rational localizations of these rings are principal ideal domains
and that \'etale covers of these rings (in the sense of Huber) are Dedekind domains.
\end{abstract}

\maketitle

\section{Introduction}

The field of $p$-adic Hodge theory has recently been transformed by a series of new geometric ideas. Central among these is the reformulation of the basic theory by Fargues and Fontaine \cite{fargues-fontaine} (see also \cite{fargues}, \cite{fargues-fontaine-durham}, \cite{part1}) in terms of vector bundles on certain noetherian schemes associated to perfect nonarchimedean fields of characteristic $p$. While these schemes are not of finite type over a field, they have certain formal properties characteristic of proper curves; for instance, their Picard groups surject canonically onto $\ZZ$.

The so-called \emph{Fargues-Fontaine curves} also admit canonical analytifications; more precisely, to each Fargues-Fontaine curve, one can functorially associate an object in Huber's category of adic spaces \cite{huber} which maps back to the original scheme in the category of locally ringed spaces. The pullback functor on vector bundles induced by this morphism is an equivalence of categories \cite[\S 8]{part1}; this constitutes a version of the GAGA principle.

One expects a similar result for coherent sheaves, but in order to build a theory of 
coherent sheaves on adic spaces, one must restrict to spaces satisfying certain noetherian hypotheses. Some care is needed because there is no analogue of the general Hilbert basis theorem for noetherian Banach rings: if $A$ is such a ring, then Tate algebras over $A$ (completion of polynomial rings over $A$ for the Gauss norm) are not known to be noetherian. One must thus consider adic spaces which locally arise from Banach rings for which the Tate algebras are all noetherian (i.e., these rings are \emph{strongly noetherian}). For such spaces, a good theory of coherent sheaves can be constructed by imitating the work of Tate and Kiehl in the case of rigid analytic spaces, as presented in \cite{bgr}; the analogue of Tate's acyclicity theorem is due to Huber
\cite[Theorem~2.5]{huber2}, while the analogue of Kiehl's glueing theorem will appear in an upcoming sequel to \cite{part1} (but see \cite[Theorem~2.7.7]{part1} for the special case of vector bundles).

In this paper, we establish the strongly noetherian property for the rings used to build the adic Fargues-Fontaine curves (Theorem~\ref{T:strongly noetherian Robba},
Theorem~\ref{T:strongly noetherian Robba2}).
These rings, which are derived from the Witt vectors over a perfect field which is complete with respect to a multiplicative norm, appear frequently in $p$-adic Hodge theory as \emph{extended Robba rings} (e.g., see \cite{part1}).
We also establish some finer properties of these rings: any rational localization is a finite direct sum of principal ideal domains (Theorem~\ref{T:strong PID}), and any \'etale covering in the sense of Huber is a finite direct sum of Dedekind domains
(Theorem~\ref{T:Dedekind}).
These statements are suggested by the origin of these rings as completions of local coordinate rings of the Fargues-Fontaine curves, which are regular noetherian schemes of dimension 1 (as shown in \cite{fargues-fontaine}).
It should be possible to go further in this direction by extending the analogy between these rings and one-dimensional affinoid algebras; see Remark~\ref{R:additional properties} for some suggestions. One can also establish an analogue of the GAGA principle for the analytification morphism between the adic and schematic Fargues-Fontaine curves;
this will also appear in a sequel to \cite{part1} (but see \cite[Theorem~8.7.7]{part1} for the special case of vector bundles).

The proof of the strongly noetherian property (Theorem~\ref{T:strongly noetherian Robba}) may be of some independent interest: it uses a form of the theory of \emph{Gr\"obner bases} which has appeared in some of our papers \cite{kedlaya-rigid-finiteness, kedlaya-liu-families} but may otherwise not be widely known to rigid analytic geometers.
For example, it can be used to recover a proof of the usual noetherian property for classical affinoid algebras distinct from the usual proof based on Weierstrass division
\cite[Theorem~5.2.6/1]{bgr}. However, it is not yet apparent to what extent Gr\"obner bases can be used to establish a general analogue of the Hilbert basis theorem for commutative nonarchimedean Banach rings.

\subsection*{Acknowledgments}
The author was supported by NSF grant DMS-1101343,
and thanks MSRI for its hospitality during fall 2014
as supported by NSF grant DMS-0932078. Thanks also to Laurent Fargues and Peter Wear for helpful discussions.

\section{Euclidean division for Witt vectors}
\label{sec:notations}

We begin by recalling the basic setup, fixing notations,
and reviewing the Euclidean division algorithm for certain rings of Witt vectors.

\begin{hypothesis}
Throughout this paper, 
let $p$ be a fixed prime, let $q$ be a power of $p$,
let $L$ be a perfect field containing $\FF_q$ which is complete with respect to the nontrivial multiplicative nonarchimedean norm $\left| \bullet \right|$,
let $E$ be a complete discretely valued field whose residue field contains $\FF_q$,
and fix a uniformizer $\varpi \in E$. Note that we allow $E$ to be of characteristic $p$;
this case is excluded in \cite{kedlaya-witt} and \cite{part1}, but in the few cases where we cite arguments that exclude this case, the replacement argument is more elementary.
\end{hypothesis}

\begin{defn}
Let $\gotho_L, \gotho_E$ denote the valuation subrings of $L,E$.
For any perfect $\FF_q$-algebra $R$, write $W(R)_E$ for the tensor product
$W(R) \otimes_{W(\FF_q)} \gotho_E$, where $W(R)$ is the usual ring of $p$-typical Witt vectors. If $E$ is of characteristic $p$, then 
$W(R)_E \cong R \otimes_{\FF_q} E$.

Define the rings
\[
A_{L,E} = W(\gotho_L)_E[[\overline{x}]: \overline{x} \in L], \qquad
B_{L,E} = A_{L,E} \otimes_{\gotho_E} E.
\]
Note that each element of $A_{L,E}$ (resp.\ $B_{L,E}$)
can be written uniquely in the form $\sum_{n \in \ZZ} \varpi^n [\overline{x}_n]$
for some $\overline{x}_n \in L$ which are zero for $n<0$ (resp.\ for $n$ sufficiently small) and bounded for $n$ large.
For $t \in [0, +\infty)$, define the ``Gauss norm'' function $\lambda_t: B_{L,E} \to \RR$ by the formula
\begin{equation} \label{eq:Gauss norm formula}
\lambda_t \left( \sum_{n \in \ZZ} \varpi^n [\overline{x}_n] \right) = \max\{
p^{-n} \left| \overline{x}_n \right|^t
\},
\end{equation}
interpreting $0^t = 0$ in the case $t=0$, so that $\lambda_0$ is the $\varpi$-adic absolute value.
\end{defn}

\begin{lemma} \label{L:Gauss norms}
For $t \in [0, +\infty)$, the function $\lambda_t$ defines a multiplicative norm on
$B_{L,E}$.
\end{lemma}
\begin{proof}
This is a straightforward consequence of the homogeneity properties of Witt vector arithmetic. See for instance \cite[\S 4]{kedlaya-witt}.
\end{proof}

For the remainder of \S\ref{sec:notations}, fix some $r>0$.

\begin{defn}
Let $A^r_{L,E}$ be the completion of $A_{L,E}$ with respect to $\lambda_r$,
and put $B^r_{L,E} = A^r_{L,E}[\varpi^{-1}]$.
Note that $A^r_{L,E}$ maps into $W(L)_E$; more precisely, if we write an arbitrary element $x \in W(L)_E$
as a $p$-adically convergent sum $\sum_{n=0}^\infty \varpi^n [\overline{x}_n]$,
then $x \in A^r_{L,E}$ if and only if $p^{-n} \left| \overline{x}_n \right|^r \to 0$ as $n \to \infty$. 
Moreover, the formula \eqref{eq:Gauss norm formula} continues to hold for $x \in A^r_{L,E}$ and $t \in [0,r]$.
Consequently, in the case $E = \QQ_p$, the rings $A^r_{L,E}, B^r_{L,E}$ coincides with the rings denoted 
$\tilde{\calR}^{\inte,r}_{L}, \tilde{\calR}^{\bd,r}_L$ in \cite{part1}.
\end{defn}

\begin{defn}
For $x = \sum_{n \in \ZZ} \varpi^n [\overline{x}_n] \in B^r_{L,E}$ 
nonzero, define the \emph{Newton polygon} of $x$
as the portion of the boundary of the convex hull of the set
\[
\bigcup_{n \in \ZZ} \{(x,y) \in \RR^2: x \leq \log_p \left| \overline{x}_n \right|,
 y \geq n\},
\]
with slopes in the range $(0,r]$. 
For $t \in (0,r]$, the \emph{multiplicity} of $t$ in (the Newton polygon of) $x$ is the height of the segment of the Newton polygon of $x$ lying on a line of slope $t$, or 0 if no such segment exists; note that this quantity is always a nonnegative integer.

For $x \in A^r_{L,E}$, we define the \emph{degree} of $x$, denoted $\deg(x)$, to be the largest $n$ realizing $\lambda_r(x) = \max_n\{p^{-n} \left| \overline{x}_n \right|^r\}$, or equivalently, the sum of the $p$-adic valuation of $x$ plus the multiplicities of all slopes of $x$. By convention, we also put $\deg(0) = -\infty$. 
\end{defn}

\begin{lemma} \label{L:additivity of slopes}
For $x_1,x_2 \in A^r_{L,E}$ nonzero and $t \in (0, r]$, the multiplicity of $t$ in (resp.\ the degree of) $x_1x_2$ is the sum of the multiplicities of $t$ in (resp.\ the degrees of) $f_1$ and $f_2$.
\end{lemma}
\begin{proof}
This follows from the multiplicative property of the norms $\lambda_t$ together with convex duality. We omit further details.
\end{proof}

The ring $A^r_{L,E}$ admits a Euclidean division algorithm as described in \cite[Lemma~2.6.3]{kedlaya-revisited}. However, we opt to give a self-contained proof for several reasons.
The level of generality in \cite{kedlaya-revisited} is at once too high (there are intended applications in which one considers somewhat smaller rings)
and too low (the field $E$ therein is forced to be of characteristic $0$)
to match our setup here.
In addition, there are a number of minor but confusing errors in the presentation
in \cite{kedlaya-revisited}; we have corrected these in the arguments that follow.
(See \cite[\S 4.2]{part1} for errata in the context of \cite{kedlaya-revisited}.)

\begin{remark} \label{R:height}
Note that for $x,y \in A^r_{L,E}$ such that $\lambda_r(x-y) < \lambda_r(x)$, we have $\deg(x) = \deg(y)$. This observation indicates that if one is willing to neglect lower-order terms, then degrees in our sense behave like the degrees of ordinary polynomials.
\end{remark}

\begin{lemma} \label{L:division algorithm1}
For $x \in A^r_{L,E}$ nonzero, there exists $\epsilon \in (0,1)$
with the following property: for any $y \in A^r_{L,E}$, we can write 
$y = z x + w $ for some $z, w \in A^r_{L,E}$ obeying the following conditions.
\begin{enumerate}
\item[(a)]
We have $\lambda_r(w) \leq \lambda_r(y)$.
\item[(b)]
If $\lambda_r(w) > \epsilon \lambda_r(y)$, then
$\deg(w) < \deg(x)$.
\end{enumerate}
\end{lemma}
\begin{proof}
Put $m = \deg(x)$ and write $x = \sum_{n=0}^\infty \varpi^n [\overline{x}_n]$.
We may then choose $\epsilon \in (0, 1)$ such that 
$\lambda_r(p) \leq \epsilon$ and
$\lambda_r(\varpi^n [\overline{x}_n]) \leq \epsilon \lambda_r(x)$ for $n > m$;
we prove the claim for this value of $\epsilon$.
Note that by the homogeneity properties of Witt vector arithmetic (see again \cite[\S 4]{kedlaya-witt}),
the first condition ensures that for any $\overline{z}_1,\dots,\overline{z}_n \in L$,
\begin{equation} \label{eq:near additivity}
\lambda_r([\overline{z}_1] \pm \cdots \pm [\overline{z}_n] - [\overline{z}_1 \pm \cdots \pm \overline{z}_n])
\leq \epsilon \max\{\lambda_r([\overline{z}_1]), \dots, \lambda_r([\overline{z}_n])\}.
\end{equation}

We define a sequence $y_0, y_1,\dots$ as follows: take $y_0 = y$,
and given $y_l = \sum_{n=0}^\infty \varpi^n [\overline{y}_{l,n}]$,
put $z_l = \sum_{n=0}^\infty \varpi^n [\overline{y}_{l,n+m}/\overline{x}_m]$
and $y_{l+1} = y_l - z_l x$.
Note that $\lambda_r(z_l) = \lambda_r(y_l)/\lambda_r(x)$, so
$\lambda_r(y_{l+1}) \leq \lambda_r(y_l) \leq \lambda_r(y)$ for all $l$.
Consequently, if for some $l$ we have either $\lambda_r(y_l) \leq \epsilon \lambda_r(y)$ or $\deg(y_l) < m$, we may take $z = z_0 + \cdots + z_{l-1}$, $w = y_l$ to achieve the desired result.

It therefore suffices to deduce a contradiction under the assumptions that $\lambda_r(y_l) > \epsilon \lambda_r(y)$ and $\deg(y_l) \geq m$ for all $l$. 
To see this, let $N_l$ be the largest value of $n$ for which 
$\lambda_r(\varpi^n [\overline{y}_{l,n}]) > \epsilon \lambda_r(y)$; note that $N_l \geq m$. Put $x' = \sum_{n=0}^{m-1} \varpi^n [\overline{x}_n]$ and write
\begin{align*}
y_{l+1} &= \sum_{n=0}^{m-1} \varpi^n [\overline{y}_{l,n}] - z_l x' - z_l \sum_{n=m+1}^\infty \varpi^n [\overline{x}_n] \\
&= \sum_{n=0}^{m-1} \varpi^n [\overline{y}_{l,n}] 
- \sum_{i=0}^{N_l-m} \sum_{j=0}^{m-1} \varpi^{i+j} [\overline{y}_{l,i+m} \overline{x}_j/\overline{x}_m]
+ *
\end{align*}
with $\lambda_r(*) \leq \epsilon \lambda_r(y)$.
Define 
$\overline{w}_n \in L$ by the following identity in $L [ T]$:
\begin{equation} \label{eq:near additivity2}
\sum_{n=0}^\infty \overline{w}_n T^n = 
\sum_{n=0}^{m-1} \overline{y}_{l,n} T^n
- \sum_{i=0}^{N_l-m} \sum_{j=0}^{m-1} (\overline{y}_{l,i+m} \overline{x}_j/\overline{x}_m) T^{i+j}.
\end{equation}
For each $n$ and each pair $(i,j)$ with $i+j=n$, we have
\[
\lambda_r(\varpi^n [\overline{y}_{l,n}]), \lambda_r(\varpi^n [\overline{y}_{l,i+m} \overline{x}_j/\overline{x}_m]) \leq \lambda_r(y).
\]
Consequently,
by applying \eqref{eq:near additivity} to the coefficients of $T^n$ in \eqref{eq:near additivity2}, then multiplying through by $\varpi^n$ and summing over $n$, we see that
\[
y_{l+1} =  \sum_{n=0}^\infty \varpi^n [\overline{w}_n] + *
\]
with $\lambda_r(*) \leq \epsilon \lambda_r(y)$.
From \eqref{eq:near additivity2},
we see that $N_{l+1} < N_l$, yielding a contradiction.
\end{proof}

\begin{prop} \label{P:division algorithm}
For $x,y \in A^r_{L,E}$ with $x \neq 0$, we can write $y = zx + w$ for some $z,w \in A^r_{L,E}$ with $\lambda_r(w) \leq \lambda_r(y)$
and $\deg(w) < \deg(x)$.
\end{prop}
\begin{proof}
Choose $\epsilon \in (0,1)$ as in Lemma~\ref{L:division algorithm1}.
We define sequences $y_0, y_1,\dots$ and $z_0, z_1,\dots$ as follows. Take $y_0 = y$.
Given $y_l$, if $\deg(y_l) < \deg(x)$, put $z_l = 0$, $y_{l+1} = y_l$.
Otherwise, apply Lemma~\ref{L:division algorithm1} to write 
$y_l = z_l x + w_l$ with $\lambda_r(w_l) \leq \lambda_r(y_l)$
and either $\lambda_r(w_l) \leq \epsilon \lambda_r(y_l)$ or $\deg(w_l) < \deg(x)$,
and put $y_{l+1} = w_l$.

We show that the sum $z = \sum_{l=0}^\infty z_l$ converges and has the desired effect.
From the construction, we have $\lambda_r(y_l) \leq \lambda_r(y)$ for all $l$.
If $z_l = 0$ for some $l$, then the sum is finite and $y-zx = y_l$, so $\lambda_r(y-zx) = \lambda_r(y_l) \leq \lambda_r(y)$ and $\deg(y-zx) = \deg(y_l) < \deg(x)$. Otherwise,
we have $\lambda_r(y_l) \leq \epsilon^l \lambda_r(y)$ for all $l$, so 
\[
\lambda_r(z_l) \leq \lambda_r(x)^{-1} \max\{\lambda_r(y_l), \lambda_r(y_{l+1})\} \to 0
\]
and the sum again converges.
We then have $y-zx = \lim_{l \to \infty} y_{l+1} = 0$, so we may take $w=0$.
\end{proof}

\begin{cor} \label{C:pid1}
The ring $A^r_{L,E}$ is a Euclidean domain for the function $\deg$, and hence a principal ideal domain.
\end{cor}

\section{The strong noetherian property}
\label{subsec:strong noetherian}

We now prove an analogue of the Hilbert basis theorem for the ring $A^r_{L,E}$.
\begin{defn}
For any commutative nonarchimedean Banach ring $A$ with norm $\left| \bullet \right|$, any nonnegative integer $n$, and any $n$-tuple $\rho = (\rho_1,\dots,\rho_n)$ of positive real numbers,
define the ring $A\{T_1/\rho_1,\dots,T_n/\rho_n\}$ as the completion of 
the ordinary polynomial ring $A[T_1,\dots,T_n]$ with respect to the weighted Gauss norm
\begin{equation} \label{eq:multi Gauss norm}
\left| \sum_{i_1,\dots,i_n=0}^\infty c_{i_1,\dots,i_n}T_1^{i_1}\cdots T_n^{i_n}\right|_\rho = 
\max_{i_1,\dots,i_n} \{\left| c_{i_1,\dots,i_n} \right| \rho_1^{i_1} \cdots \rho_n^{i_n}\}.
\end{equation}
We may view $A\{T_1/\rho_1,\dots,T_n/\rho_n\}$ as the subring of
$A \llbracket T_1,\dots,T_n \rrbracket$ consisting of those series
$\sum_{i_1,\dots,i_n=0}^\infty c_{i_1,\dots,i_n}T_1^{i_1}\cdots T_n^{i_n}$
for which $\left| c_{i_1,\dots,i_n} \right| \rho_1^{i_1} \cdots \rho_n^{i_n} \to 0$
as $i_1 + \cdots + i_n \to \infty$, with the norm again given by \eqref{eq:multi Gauss norm}.
Note that if $\left| \bullet \right|$ is multiplicative, then so is $\left| \bullet \right|_\rho$ (Gauss's lemma; see \cite[Lemma~1.7]{kedlaya-witt}).
\end{defn}

One would like to know that $A\{T_1/\rho_1,\dots,T_n/\rho_n\}$ is noetherian whenever $A$ is, but this is only known under somewhat restrictive hypotheses, e.g., when $A$ is a nonarchimedean field \cite[Theorem~5.2.6/1]{bgr}. 
Over the course of \S\ref{subsec:strong noetherian}, we will prove the following theorem,
which answers a question of Fargues \cite{fargues}
by proving that $A^r_{L,E}$ is \emph{strongly noetherian} in the sense of Huber. This means that Huber's theory of adic spaces, as developed in \cite{huber}, applies to this ring; we will pursue this point in an upcoming sequel to \cite{part1}.

\begin{theorem} \label{T:strongly noetherian Robba}
For $r>0$, view $A^r_{L,E}$ as a Banach ring using the norm $\lambda_r$. Then for any nonnegative integer $n$ and any $\rho_1,\dots,\rho_n>0$,
the ring $R = A^r_{L,E}\{T_1/\rho_1,\dots,T_n/\rho_n\}$ is noetherian.
\end{theorem}

Our approach to the proof relies on some standard ideas from the theory of Gr\"obner bases; indeed, it can be used to give an alternate proof of \cite[Theorem~5.2.6/1]{bgr}.
We start with the underlying combinatorial construction.

\begin{hypothesis}
For the remainder of \S\ref{subsec:strong noetherian}, retain notation as in 
Theorem~\ref{T:strongly noetherian Robba}, let $H$ be an ideal of $R$,
and let $I = (i_1,\dots,i_n)$ and $J = (j_1,\dots,j_n)$ (and subscripted versions thereof,
such as $I_k = (i_{k,1},\dots,i_{k,n})$) 
denote elements of the additive monoid $\ZZ_{\geq 0}^n$ of $n$-tuples of nonnegative integers.
\end{hypothesis}

\begin{defn}
We equip $\ZZ_{\geq 0}^n$ with the componentwise partial order $\leq$,
for which $I \leq J$ if and only if $i_k \leq j_k$ for $i=1,\dots,n$. 
This partial order is a \emph{well-quasi-ordering}: any infinite sequence contains an infinite nondecreasing subsequence.

We also equip $\ZZ_{\geq 0}^n$ with the \emph{graded lexicographic} total order $\preceq$, for which $I \prec J$ if 
either $i_1 + \cdots + i_n < j_1 + \cdots + j_n$, 
or
$i_1 + \cdots + i_n = j_1 + \cdots + j_n$
and there exists $k \in \{1,\dots,n\}$ such that $i_l = j_l$ for $l<k$ and $i_k < j_k$.
Since $\preceq$ is a refinement of $\leq$, it is a well-ordering.
\end{defn}

\begin{remark}
In commutative algebra, the only critical properties of $\preceq$ are that it is a well-ordering and that it refines $\leq$. In some cases (such as ours), it is also important that for any $I$, there are only finitely many $J$ with $J \preceq I$. In any case, there are many options for $\preceq$ with similar properties,
giving rise to many different \emph{term orderings} which are relevant for practical applications. See for instance \cite{eisenbud}.
\end{remark}

We next define a notion of \emph{leading terms} for elements of $R$. Note that a similar construction appears already in \cite{kedlaya-rigid-finiteness}.
\begin{defn}
For $x = \sum_I x_I T^I \in R$ nonzero,
define the \emph{leading index} of $x$ to be the index $I$ which is maximal under $\preceq$ for the property that $\left| x_I T^I \right|_\rho = \left| x \right|_\rho$,
and define the \emph{leading coefficient} of $x$ to be the corresponding value of $x_I$.
\end{defn}

We can now construct an analogue of a Gr\"obner basis for the ideal $H$.
\begin{defn}
For each $I$, let $d_I$ be the smallest possible degree of the leading coefficient of an element of $H$ with leading index $I$, or $+\infty$ if no such element exists.
Note that if $I_1 \leq I_2$, then $d_{I_2} \leq d_{I_1}$.

Since $\ZZ_{\geq 0}^n$ is well-quasi-ordered under $\leq$,
the set of $I$ for which $d_I < +\infty$ contains only finitely many minimal elements with respect to $\leq$. Consequently, the set of possible finite values of $d_I$ is bounded above, and hence is finite. For each nonnegative integer $d$, let $S_d$ be the set of $I$ which are minimal with respect to $\leq$ for the property that $d_I = d$; then $S_d$ is finite for all $d$ and empty for all but finitely many $d$. Let $S$ be the union of the $S_d$. 
For each $I \in S$, choose $x_I \in H \setminus \{0\}$
with leading index $I$ and leading coefficient of degree $d_I$.
\end{defn}

We claim that the finite set $\{x_I: I \in S\}$ generates the ideal $H$.
As in the proof of Proposition~\ref{P:division algorithm}, we first establish a certain approximate version of this statement, using an iterative construction and a proof by contradiction based on well-ordering properties.
\begin{lemma} \label{L:reduce size with generators}
There exists $\epsilon \in (0,1)$ with the following property: for each $y \in H$,
there exist $a_I \in R$ for $I \in S$ such that
$\left| a_I \right|_\rho \left| x_I \right|_\rho \leq \left| y \right|_\rho$ and
$\left| y - \sum_{I \in S} a_I x_I \right|_\rho \leq \epsilon \left|y \right|_\rho$.
\end{lemma}
\begin{proof}
Write $x_I = \sum_{J} x_{I,J} T^{J}$ and let $c_I = x_{I,I}$ be the leading coefficient of $x_I$. Let $\epsilon$ be the maximum of $\left| x_{I,J} T^J \right|_\rho / \left| c_I T^I \right|_\rho$ over all $I \in S$ and $J$ for which $I \prec J$, provided that this maximum is positive (and hence belongs to $(0,1)$); otherwise, choose any $\epsilon \in (0,1)$. We prove the claim for this value of $\epsilon$.

We define $y_l \in H$, $a_{l,I} \in R$ for $l=0,1,\dots$ and $I \in S$ as follows.
Put $y_0 = y$. Given $y_l = \sum_{J} y_{l,J} T^{J}$,
if $\left| y_l \right|_\rho \leq \epsilon \left| y \right|_\rho$, put $a_{l,I} = 0$
and $y_{l+1} = y_l$.
Otherwise, $y_l$ is nonzero, so it has a leading index $J_l$. By construction, we can find an index $I_l \in S$ such that $I_l \leq J_l$ and $d_{I_l} = d_{J_l}$.
Apply Proposition~\ref{P:division algorithm} to write $y_{l,J_l} = z_l c_{I_l} + w_l$ 
for some $z_l, w_l \in A^r_{L,E}$ with $\left| w_l \right| \leq \left| y_{l,J_l} \right|$
and $\deg(w_l) < \deg(c_{I_l}) = d_{I_l}$.
Put
\[
a_{l,I} = \begin{cases} z_l T^{J_l - I_l} & (I = I_l) \\ 0 & (I \neq I_l), 
\end{cases}
\qquad
y_{l+1} = y_l- a_{l,I_l} x_{I_l}.
\]
If $\left| y_l \right|_\rho \leq \epsilon \left| y \right|_\rho$ for some $l$, then the sums
$a_I = \sum_{l=0}^\infty a_{l,I}$ are finite and have the desired effect. It thus suffices to derive a contradiction under the assumption that 
$\left| y_l \right|_\rho > \epsilon \left| y \right|_\rho$ for all $l$.

Define the \emph{$\epsilon$-support} of $y_l$ to be the finite set $E_l$ consisting of those $J$ for which $\left| y_{l, J} T^J \right|_\rho > \epsilon \left| y \right|_\rho$;
in particular, $J_l \in E_l$.
By virtue of our choice of $\epsilon$, $E_l$ and $E_{l+1}$ agree for all indices $J$ for which $J_l\prec J$.
In particular, since $E_0$ is finite, we can choose $J_+$ for which $J \preceq J_+$ for all $J \in E_0$, and then $J \preceq J_+$ for $J \in E_{l}$ for all $l$.

The set $\{J \in \ZZ_{\geq 0}^n: J \preceq J_+\}$ is finite, so for some $l_0$, every index which occurs as $J_l$ for a single $l \geq l_0$ occurs for infinitely many such $l$. Let $J$ be the largest such index with respect to $\preceq$, choose some $l \geq l_0$ for which $J = J_l$, and let $l'$ be the smallest value greater than $l$ for which $J_l = J_{l'}$; then $J_k \prec J$ for $l< k<l'$.
By the choice of $\epsilon$, we have 
\[
\left| y_{k,J} T^J - y_{k+1,J} T^J \right|_\rho \leq \epsilon \left|y \right|_\rho \qquad (k = l+1, \dots, l'-1)
\]
and hence
\[
\left| y_{l+1,J} T^J - y_{l',J} T^J \right|_\rho \leq \epsilon \left|y \right|_\rho < \left| y_{l',J} T^J\right|_\rho.
\]
But now
\[
\deg(y_{l+1,J}) < d_J, \left| y_{l+1,J} - y_{l',J} \right| < \left| y_{l',J} \right|
\]
and by Remark~\ref{R:height} this yields $\deg(y_{l',J}) < d_J$, contradicting the definition of $d_J$.
\end{proof}

We now finish as in the proof of Proposition~\ref{P:division algorithm}.

\begin{lemma}
The finite set $\{x_I: I \in S\}$ generates the ideal $H$.
Consequently, Theorem~\ref{T:strongly noetherian Robba} holds.
\end{lemma}
\begin{proof}
Choose $\epsilon \in (0,1)$ as in Lemma~\ref{L:reduce size with generators}.
For $y \in H$, define sequences $y_0, y_1,\dots$ and $a_{0,I}, a_{1,I}, \dots$ for $I \in S$ as follows: put $y_0 = y$,
and given $y_l$,
apply Lemma~\ref{L:reduce size with generators}
to construct $a_{l,I} \in R$ for $I \in S$ such that
$\left| a_{l,I} \right|_\rho \left| x_I \right|_\rho \leq \left| y_l \right|_\rho$ and
$\left| y_l - \sum_{I \in S} a_{l,I} x_I \right|_\rho \leq \epsilon \left|y_l \right|_\rho$,
then put $y_{l+1} = y_l - \sum_{I \in S} a_{l,I} x_I$. 
By construction, $\left| y_l \right|_\rho \leq \epsilon^l \left| y_l \right|_\rho$,
so the sequence $\{y_l\}_{l=0}^\infty$ converges to zero and the sums 
$a_I = \sum_{l=0}^\infty a_{l,I}$ converge to limits satisfying
$y = \sum_{I \in S} a_I x_I$.
\end{proof}

\section{Some additional rings}
\label{sec:other rings}

We next define the rings that appear directly in the study of the adic spaces associated to Fargues-Fontaine curves, and use Theorem~\ref{T:strongly noetherian Robba} to extend the strong noetherian property to these rings.

\begin{hypothesis}
Throughout \S\ref{sec:other rings}, let $I = [s,r]$ be a closed subinterval of $(0, +\infty)$.
\end{hypothesis}

\begin{defn}
Define $\lambda_I = \max\{\lambda_s, \lambda_r\}$; by Lemma~\ref{L:Gauss norms}, this is a power-multiplicative norm on $B_{L,E}$.
Let $B^I_{L,E}$ be the completion of $B_{L,E}$ with respect to $\lambda_I$. 
Let $B^{I,+}_{L,E}$ denote the subring of $x \in B^I_{L,E}$ for which $\lambda_I(x) \leq 1$.
\end{defn}

\begin{remark} \label{R:identify rings}
In the case $E = \Qp$, the ring $B^I_{L,E}$ coincides with the ring $\tilde{\calR}^I_L$ of \cite{part1}; in general, it appears under the notation $B_I$ in \cite{fargues-fontaine}.
\end{remark}

The following may be considered an analogue of the Hadamard three circles inequality
(compare \cite[Lemma~4.2.3]{part1}).
\begin{lemma} \label{L:Hadamard}
For $t_1, t_2 \in I$ and $c \in [0,1]$, put $t = t_1^c t_2^{1-c}$. Then for all $x \in B^I_{L,E}$,
\[
\lambda_t(x) \leq \lambda_{t_1}(x)^c \lambda_{t_2}(x)^{1-c}.
\]
\end{lemma}
\begin{proof}
By continuity, it suffices to check the inequality for $x \in B_{L,E}$.
From the shape of the formula \eqref{eq:Gauss norm formula}, we may further reduce to the case where $x = \varpi^n [\overline{x}_n]$ for some $n \in \ZZ$, $\overline{x}_n \in L$. But in this case, the desired inequality becomes an equality.
\end{proof}
\begin{cor}\label{C:Banach to Frechet}
We have $\lambda_{I} = \sup\{\lambda_t: t \in I\}$.
\end{cor}
\begin{cor}
For any closed subinterval $I'$ of $I$, there is a natural injective map $B^I_{L,E} \to B^{I'}_{L,E}$.
\end{cor}
\begin{proof}
The existence of the map follows from Corollary~\ref{C:Banach to Frechet}.
To check injectivity, suppose $x \in B^I_{L,E}$ maps to zero in $B^{I'}_{L,E}$. Then $\lambda_t(x) = 0$ for all $t \in I'$, but by Lemma~\ref{L:Hadamard} this also implies
$\lambda_t(x) = 0$ for all $t$ in the interior of $I$. By continuity, this implies $\lambda_t(x) = 0$ for all $t \in I$, whence $x = 0$.
\end{proof}

\begin{defn} \label{D:Newton polygon extended}
For $x \in B^I_{L,E}$ nonzero, define the \emph{Newton polygon} of $x$ by choosing some $x' \in B_{L,E}$ with $\lambda_t(x'-x) < \lambda_t(x)$ for all $t \in I$, forming the Newton polygon of $x'$, then discarding segments corresponding to slopes not in $I$. Note that this  construction does not depend on the choice of $x'$, and inherits the multiplicativity properties from the corresponding definition for $B_{L,E}$ (Lemma~\ref{L:additivity of slopes}). We define the \emph{multiplicity} of slopes as before, and the \emph{degree} of $x$ as the sum of all multiplicities, which is again a nonnegative integer.
\end{defn}

\begin{lemma} \label{L:unit}
A nonzero element $x \in B^I_{L,E}$ is a unit if and only if its degree is $0$.
\end{lemma}
\begin{proof}
If $x$ is a unit, it must have degree $0$ by the multiplicativity property of Newton polygons (Definition~\ref{D:Newton polygon extended}). Conversely, if $x$ has degree 0, then for some $n \in \ZZ$, $\overline{x}_n \in L^\times$ we have $\lambda_t(x - \varpi^n [\overline{x}_n]) < \lambda_t(\varpi^n [\overline{x}_n])$ for all $t \in I$, so we may compute an inverse of $x$ using a convergent geometric series.
\end{proof}

\begin{lemma} \label{L:Robba localizations}
Choose $\overline{z} \in L$ with $\left| \overline{z} \right| = c \in (0,1)$. Then for $\rho \in (0,1)$, we have isomorphisms of Banach rings
\begin{align*}
B^I_{L,E}\{T/\rho\}/(T-[\overline{z}]) &\cong B^{I'}_{L,E}, \qquad &I' = I \cap [\log_c \rho, +\infty), \\
B^I_{L,E}\{T/\rho^{-1}\}/(T-[\overline{z}^{-1}]) &\cong B^{I''}_{L,E}, \qquad &I'' = I \cap (0, \log_c \rho], \\
A^r_{L,E}\{T\}/(pT-[\overline{z}^n]) &\cong B^{I'''}_{L,E}, \qquad & I''' =[- n^{-1} \log_c p, r],
\end{align*}
interpreting $B^*_{L,E}$ as $0$ if $*$ is empty.
Moreover, in case $\rho \in p^{\QQ}$,
the integral closures of the images of $B^{I,+}_{L,E}$ in $B^{I'}_{L,E}, B^{I''}_{L,E}$
are respectively $B^{I',+}_{L,E}, B^{I'',+}_{L,E}$.
\end{lemma}
\begin{proof}
We check only the first case in detail, the other cases being similar.
Put $t_0 = \log_c \rho$.
If $I' = \emptyset$, then $T - [\overline{z}] = [\overline{z}](1 - [\overline{z}^{-1}] T)$ and $\left| [\overline{z}^{-1}] T \right|_\rho < 1$,
so $T- [\overline{z}]$ is a unit in $B^I_{L,E}/\{T/\rho\}$ and so 
both sides of the desired equality are zero.
We may thus assume hereafter that $I' \neq \emptyset$,
so that there is a well-defined map $B^I_{L,E}\{T/\rho\} \to B^{I'}_{L,E}$ taking $T$ to $[\overline{z}]$.

For $x,y \in B^I_{L,E}\{T/\rho\}$ with $y = (T-[\overline{z}]) x$,
applying the multiplicative property of Gauss norms and then taking suprema yields
$\left| y \right|_\rho \geq \left| x \right|_\rho$
(compare \cite[Lemma~2.8.8]{part1}).
This means that multiplication by $T-[\overline{z}]$ is a strict injective endomorphism of $B^I_{L,E}\{T/\rho\}$. In particular, the ideal $(T - [\overline{z}])$ is closed, so $B^I_{L,E}\{T/\rho\}/(T-[\overline{z}])$ is a Banach ring. 

Next, suppose that $y = \sum_{n=0}^\infty y_n T^n \in B^I_{L,E}\{T/\rho\}$ maps to zero in $B^{I'}_{L,E}$. Put
\[
x_n = -\sum_{i=0}^n y_i [\overline{z}]^{i-n-1},
\]
so that in $B^I_{L,E}\llbracket T \rrbracket$ we have $y = (T-[\overline{z}]) x$ for
$x = \sum_{n=0}^\infty x_n T^n$. For $t \in I$ with $t < t_0$,
we see as above that $T-[\overline{z}]$ is invertible in $B^{[t,t]}_{L,E}\{T/\rho\}$
and so $\rho^n \lambda_t(x_n) \to 0$. For $t \in I$ with $t \geq t_0$,
we have $\lambda_t([\overline{z}]) \leq \rho$ and so 
$B^{[t,t]}_{L,E}\{T/\rho\}/(T-[\overline{z}]) \cong B^{[t,t]}_{L,E}$;
hence $\rho^n \lambda_t(x_n) \to 0$ again. 
Using Corollary~\ref{C:Banach to Frechet},
we conclude that
$x \in B^I_{L,E}\{T/\rho\}$, so the map $B^I_{L,E}\{T/\rho\}/(T-[\overline{z}]) \to B^{I'}_{L,E}$ is injective.

Next, put $x = \varpi^n [\overline{x}_n]$ for some $n \in \ZZ$, $\overline{x}_n \in L$.
Let $j$ be the smallest nonnegative integer such that $c^{-j} \left| \overline{x}_n \right| \geq 1$ and take $y = \varpi^n [\overline{x}_n \overline{z}^{-j}]$.
For $t \in I$ with $t \leq t_0$, we have $\rho^j \lambda_t(y) \leq \rho^j \lambda_{t_0}(y) = \lambda_{t_0}(x)$.
For $t \in I$ with $t > t_0$, in case $j=0$ we obviously have $\lambda_t(y) = \lambda_t(x)$;
otherwise, we have $c^{-j+1} \left| \overline{x}_n \right| < 1$ and so
$\rho^j \lambda_t(y) < c^{t_0-t} \lambda_{t_0}(x)$.
For $z = y T^j \in B^I_{L,E}\{T/\rho\}$, we therefore have
\begin{equation} \label{eq:Robba localization lift}
\left| z \right|_\rho \leq c^{t_0-t} \lambda_{I'}(x).
\end{equation}
By \eqref{eq:Gauss norm formula}, we may then lift any $x \in B^{I'}_{L,E}$ to $z \in B^I_{L,E}\{T/\rho\}$ so that \eqref{eq:Robba localization lift} remains true.
This implies that $B^I_{L,E}\{T/\rho\}/(T-[\overline{z}]) \cong B^{I'}_{L,E}$
is strict surjective.

To conclude, it is sufficient to check that if $\rho \in p^\QQ$, then for any $x = \varpi^i [\overline{x}_n]$ with $\lambda_{I'}(x) = 1$, we can lift some power of $x$ to $z \in B^{I}_{L,E}\{T/\rho\}$ with $\left| z \right|_\rho = 1$.
Set notation as above. If $j=0$, then 
\[
\lambda_{I'}(x) = \lambda_s(x) = \lambda_s(y) = \lambda_I(y) = \left| z \right|_\rho.
\]
If $j>0$, then $\lambda_{I'}(x) = \lambda_{t_0}(x) = p^{-n} \left| \overline{x}_n \right|^{t_0}$. Since $\lambda_{I'}(x) = 1$ and $\rho \in p^{\QQ}$, after raising $x$ to a suitable power we have $\left| \overline{x}_n \overline{z}^{-j} \right| = 1$, so
\[
\left| z \right|_\rho = \rho^j p^{-n} = \lambda_{t_0}(x) = \lambda_{I'}(x).
\]
This completes the proof.
\end{proof}

By combining Theorem~\ref{T:strongly noetherian Robba} and
Lemma~\ref{L:Robba localizations}, we obtain the following result.
\begin{theorem} \label{T:strongly noetherian Robba2}
View $B^I_{L,E}$ as a Banach ring using the norm $\lambda_I$. Then for any nonnegative integer $n$ and any positive real numbers $\rho_1,\dots,\rho_n$,
the ring $B^I_{L,E}\{T_1/\rho_1,\dots,T_n/\rho_n\}$ is noetherian.
\end{theorem}

\begin{remark}
The fact that the rings $B^I_{L,E}\{T_1/\rho_1,\dots,T_n/\rho_n\}$ are noetherian for $\rho_1 = \cdots = \rho_n = 1$ means that $B^I_{L,E}$ is strongly noetherian in the sense of Huber. However, we do not know how to deduce this directly from the restricted form of Theorem~\ref{T:strongly noetherian Robba} in which one only allows $\rho_1 = \cdots = \rho_n = 1$: we need to allow arbitrary $\rho$ in order to fix the left endpoint of the interval $I$ using Lemma~\ref{L:Robba localizations}.
We also do not know how to give a direct proof of Theorem~\ref{T:strongly noetherian Robba2} in the style of the proof of Theorem~\ref{T:strongly noetherian Robba} except in the case where $I = [r,r]$ consists of a single point, in which case $\lambda_I = \lambda_r$ is again multiplicative.
\end{remark}

\begin{remark}
One can make sense of $B^{[0,r]}_{L,E}$ by identifying it with $B^r_{L,E} = A^r_{L,E}[\varpi^{-1}]$. This gives a Banach ring for the norm $\max\{\lambda_0, \lambda_r\}$; note that $\lambda_0$ is the $\varpi$-adic absolute value. Of course $B^{[0,r]}_{L,E}$ is noetherian because $A^r_{L,E}$ is. However, due to the mismatch of topologies, we do not know how to prove that $B^{[0,r]}_{L,E}$ is strongly noetherian.
\end{remark}

\begin{remark} \label{R:regular excellent}
It is natural to ask whether the rings $B^I_{L,E}\{T_1/\rho_1,\dots,T_n/\rho_n\}$ are not only noetherian, but also regular and excellent. We have not considered this question further.
\end{remark}

\begin{remark}
One can also make sense of $B^{[r,+\infty]}_{L,E}$ by rescaling the Gauss norms, e.g., by setting
\[
\lambda_t \left( \sum_{n \in \ZZ} \varpi^n [\overline{x}_n] \right) = \sup\{
p^{-n/(1+t)} \left| \overline{x}_n \right|^{t/(1+t)}
\}
\]
so that
\[
\lambda_\infty \left( \sum_{n \in \ZZ} \varpi^n [\overline{x}_n] \right) = \sup\{
\left| \overline{x}_n \right|
\}.
\]
However, the resulting ring can be shown to be nonnoetherian, by exploiting the existence of elements with infinitely many distinct slopes in their Newton polygons (or equivalently, the fact that the maxima have become suprema).
\end{remark}

\section{A descent construction}
\label{sec:descent}

Before continuing, we record a descent argument which will allow us to freely enlarge the field $L$ in what follows.

\begin{convention}
We adopt conventions concerning Banach rings, adic Banach rings, Gel'fand spectra, and adic spectra as in \cite{part1}. In particular, we write $\calM(R)$ for the Gel'fand spectrum of the Banach ring $R$ and $\Spa(R,R^+)$ for the adic spectrum of the adic Banach ring $(R,R^+)$. Each $\beta \in \calM(R)$ is a multiplicative seminorm on $R$;
we write $\calH(\beta)$ for the completion of $\Frac(R/\ker(\beta))$ with respect to the induced multiplicative norm.
\end{convention}

\begin{hypothesis} \label{H:extension}
Throughout \S\ref{sec:descent}, let $L'$ be a perfect overfield of $L$ which is complete with respect to a multiplicative nonarchimedean norm extending the norm on $L$.
\end{hypothesis}

\begin{lemma} \label{L:tensor product}
The following statements hold.
\begin{enumerate}
\item[(a)]
The tensor product seminorms on $L' \otimes_L L'$ and $L' \otimes_L L' \otimes_L L'$
are power-multiplicative.
\item[(b)]
The simplicial exact sequence
\[
0 \to L \to L' \otimes_L L' \to L' \otimes_L L' \otimes_L L'
\]
is almost optimal; 
that is, the quotient and subspace seminorms at each point coincide.
Consequently, we may complete the tensor product to obtain another almost optimal exact sequence.
\end{enumerate}
\end{lemma}
\begin{proof}
See \cite[Remark~3.1.6]{part1}.
\end{proof}

Using Lemma~\ref{L:tensor product}, we obtain a descent property for ideals in $W(\gotho_{L})_E$.

\begin{lemma} \label{L:primitive descent}
Equip $R = L' \widehat{\otimes}_L L'$ with the tensor product seminorm.
Let $z$ be an element of $W(\gotho_{L'})_E$ with the property that 
for any $\beta \in \calM(R)$, the two images of $z$ in $W(\gotho_{\calH(\beta)})_E$ generate the same ideal and their ratio maps to $1$ under $W(\gotho_{\calH(\beta)})_E \to W(\kappa_{\calH(\beta)})$. Then $z$ factors as a unit times an element of $W(\gotho_L)_E$.
\end{lemma}
\begin{proof}
Let $\kappa_L,\kappa_{L'}$ be the residue fields of $L, L'$.
Let $\gotho_R$ be the subring of $R$ consisting of elements of norm at most 1.
Let $\gothm_R$ be the ideal of $\gotho_R$ consisting of elements of norm strictly less than 1. Put $\kappa_R = \gotho_R/\gothm_R$.
By Lemma~\ref{L:tensor product}(a) and \cite[Theorem~2.3.10]{part1},
the tensor product seminorm on $R$ can be computed as the supremum over $\calM(R)$.
Consequently, we have a canonical isomorphism $\kappa_R \cong \kappa_{L'} \otimes_{\kappa_L} \kappa_{L'}$.

Let $\iota_1, \iota_2$ denote the two maps $L' \to R$ and also the induced maps $W(L')_E \to W(R)_E$. 
Put $q_0 = \iota_1(z)/\iota_2(z) \in W(R)_E$. By considering Newton polygons in $W(\gotho_{\calH(\beta)})_E$ for each $\beta \in \calM(R)$, we see that
in fact $q_0 \in W(\gotho_R)_E^\times$; using the condition on ratios, we see that moreover
$q_0 - 1 \in \ker(W(\gotho_R)_E \to W(\kappa_R))$.

Define the submultiplicative norm $\lambda_1$ on $W(\gotho_R)_E$ using the formula
\eqref{eq:Gauss norm formula}. We can then choose $\epsilon \in (0,1)$ such that
$\lambda_1(q_1-1) \leq \epsilon^2$.
We construct sequences $u_1=1,u_2,\ldots \in W(\gotho_{L'})_E^\times$ and $q_1, q_2, \ldots \in W(\gotho_R)_E^\times$ as follows: given $q_l$, apply Lemma~\ref{L:tensor product}(b) to construct $v_l \in W(\gotho_{L'})_E$ with $\lambda_1(v_l) \leq \epsilon^{-1} \lambda_1(q_l)$ and $\iota_1(v_l) - \iota_2(v_l) = q_l$,
then put $u_{l+1} = u_l (1 + v_l)$ and $q_{l+1} = \iota_1(z/u_{l+1})/\iota_2(z/u_{l+1})$.
We then have $\lambda_1(q_l-1) \leq \epsilon^{l+1}$ and hence $\lambda_1(v_l) \leq \epsilon^l$, so the $q_l$ converge to 1 and the $u_l$ converge to a limit $u \in W(\gotho_{L'})_E^\times$ for which $z/u \in W(\gotho_L)_E$.
\end{proof}

\begin{remark}
It was pointed out by a referee that Lemma~\ref{L:primitive descent} fails
without the ratio condition. For instance, let $L$ be the completed perfect closure of $\FF_p((t))$ and take $L' = L(t^{1/2})$, $z = [t^{1/2}]$.
\end{remark}

\section{Primitive elements of degree 1}
\label{sec:primitive}

We next focus attention on those elements of $W(\gotho_L)_E$ which behave like monic linear polynomials in the variable $p$. These elements control much of the algebra and geometry of the rings we are considering. In particular, they give rise to a deformation retraction on 
$\calM(B^I_{L,E})$ as described in \cite{kedlaya-witt}.

\begin{defn} \label{D:primitive of degree 1}
We say that $z = \sum_{n=0}^\infty \varpi^n [\overline{z}_n] \in W(\gotho_L)_E$ is \emph{primitive of degree $1$} if $\overline{z}_0 \in \gothm_L \setminus \{0\}$
and $\overline{z}_1 \in \gotho_L^\times$.
For example, $\varpi - [\overline{u}]$ is primitive of degree $1$ for any $\overline{u} \in \gothm_L \setminus \{0\}$.
For $z$ primitive of degree 1, the \emph{slope} of $z$ is the unique slope $r$ in the Newton polygon of $z$; the ring $W(\gotho_L)_E[\varpi^{-1}]/(z)$ is a perfectoid field under the quotient norm induced by $\lambda_r$ in case $E$ is of characteristic 0
\cite[Theorem~3.5.3]{part1}, and is equal to $L$ in case $E$ is of characteristic $p$.
\end{defn}

\begin{defn}
For $z \in W(\gotho_{L})_E$ primitive of degree 1 with slope $r \in I$,
let $H(z,\rho)$ be the quotient norm on $B^I_{L,E}\{T/(p^{-1} \rho)\}/(T-z)$
(interpreted as $B^I_{L,E}/(z)$ for $\rho=0$)
for the Gauss extension of $\lambda_r$.
As in \cite[Theorem~5.11]{kedlaya-witt}, this norm is multiplicative.
\end{defn}

\begin{lemma} \label{L:find primitive point}
Let $I$ be a closed interval in $(0, +\infty)$.
For any $\beta \in \calM(B^I_{L,E})$, 
there exist a perfect overfield $L'$ of $L$ complete with respect to a multiplicative nonarchimedean norm extending the one on $L$
and some $\overline{u} \in \gothm_{L'} \setminus \{0\}$ such that the restriction of $H(\varpi - [\overline{u}], 0)$
to $B^I_{L,E}$ equals $\beta$.
\end{lemma}
\begin{proof}
Let $F'$ be a completed algebraic closure of $\calH(\beta)$.
Let $L'$ be the perfect field corresponding to $F'$ under the perfectoid correspondence
\cite[Theorem~3.5.3]{part1} (taking $L' = F'$ if $E$ is of characteristic $p$); recall that $L'$ may be identified set-theoretically with the the inverse limit of $F'$ under the $p$-power map. We may then take $\overline{u}$ to be a coherent sequence of $p$-power roots of $p$ in $F'$.
\end{proof}

\begin{defn} \label{D:homotopy}
Let $I$ be a closed interval in $(0, +\infty)$.
For any $\beta \in \calM(B^I_{L,E})$ and $\rho \in [0,1]$, we may define a point
$H(\beta,\rho) \in \calM(B^I_{L,E})$ by choosing $L', \overline{u}$ as in
Lemma~\ref{L:find primitive point} and taking $H(\beta,\rho)$ to be the restriction of
$H(\varpi-[\overline{u}],\rho)$.
Note that $H(\beta,0) = \beta$ while $H(\beta,1) = \lambda_r$ for $r$ equal to the slope of $\varpi - [\overline{u}]$.
 As in \cite[Theorem~7.8]{kedlaya-witt},
this construction does not depend on $L', \overline{u}$ and defines a continuous map
$H: \calM(B^I_{L,E}) \times [0,1] \to \calM(B^I_{L,E})$ satisfying
\begin{equation} \label{eq:composite homotopy}
H(H(\beta,\rho),\sigma) = H(\beta, \max\{\rho,\sigma\})
\qquad
(\beta \in \calM(B^I_{L,E});  \rho, \sigma \in [0,1]).
\end{equation}
We define the \emph{radius} of $\beta$ to be the largest $\rho \in [0,1]$ for which 
$H(\beta,\rho) = \beta$.
\end{defn}

\begin{defn} \label{D:join}
For $\beta, \gamma \in \calM(B^I_{L,E})$, let $\beta \wedge \gamma$ be the element
$H(\beta, \rho) \in \calM(B^I_{L,E})$ for the smallest value of $\rho \in [0,1]$ such that
$H(\beta,\rho) = H(\gamma,\sigma)$ for some $\sigma \in [0,1]$. Using \eqref{eq:composite homotopy}, we may see that the operation $\wedge$ is idempotent, symmetric, and associative; we call $\beta \wedge \gamma$ the \emph{join} of $\beta$ and $\gamma$ in $\calM(B^I_{L,E})$.
By reducing to the setting of Lemma~\ref{L:find primitive point}, 
one verifies that the  map $\beta \wedge \bullet: \calM(B^I_{L,E}) \to \calM(B^I_{L,E})$ is again continuous; its image may be identified with $[\rho, 1]$ for $\rho$ the radius of $\beta$ via the map $H(\beta, \bullet)$.
\end{defn}

\begin{lemma} \label{L:complementary components}
Fix $\beta \in \calM(B^I_{L,E})$. For $\gamma, \delta \in \calM(B^I_{L,E}) \setminus \{\beta\}$, write $\gamma \sim \delta$ if one of the following conditions holds:
\begin{enumerate}
\item[(a)]
$\beta \wedge \gamma \neq \beta$, $\beta \wedge \delta \neq \beta$; or
\item[(b)]
$\beta \wedge \gamma = \beta$, $\beta \wedge \delta = \beta$, and $\gamma \wedge \delta \neq \beta$.
\end{enumerate}
Then $\sim$ is an equivalence relation, and the equivalence classes under $\sim$ are
open subsets of $\calM(B^I_{L,E})$.
In particular, these equivalence classes constitute the connected components of $\calM(B^I_{L,E}) \setminus \{\beta\}$, and each of these components is also path-connected.
\end{lemma}
\begin{proof}
Let $\rho$ be the radius of $\beta$. Then the unique equivalence class described in (a)
is the inverse image of $(\rho,1] \subset [\rho, 1]$ under the map $\beta \wedge \bullet$, and therefore is open. Now choose an equivalence class described in (b), choose an element $\gamma$ thereof, and let $\rho'$ be the radius of $\gamma$; then the equivalence class
is the inverse image of $[\rho', \rho) \subset [\rho', 1]$ under the map $\gamma \wedge \bullet$, and therefore is open.
\end{proof}

\begin{remark} \label{R:geometric analogy}
The previous constructions are part of a broad analogy between the structure of $\calM(B^I_{L,E})$ and that of the closed unit disc $\calM(K\{T\})$.
The latter has the structure of an inverse limit of trees rooted at the Gauss point (the point corresponding to the Gauss norm on $K\{T\}$); see \cite[\S 1.4, Figure~1]{baker-rumely} for a picture. 
The analogue of Definition~\ref{D:homotopy} is the continuous map $H: \calM(K\{T\}) \times [0,1] \to \calM(K\{T\})$ constructed in \cite[Theorem~2.5]{kedlaya-witt},
which maps $(\beta,\rho)$ to the generic point of the closed disc of radius $\rho$ containing $\beta$; this gives a deformation retraction of $\calM(K\{T\})$ onto the Gauss point. The analogue of Definition~\ref{D:join} is the map $\beta \wedge \bullet$ taking
$\gamma \in \calM(K\{T\})$ to the generic point of the smallest closed disc containing both $\beta$ and $\gamma$.

This analogy can be extended somewhat further. For example, by \cite[Theorem~2.11]{kedlaya-witt}, for $\beta, \gamma \in \calM(K\{T\})$, $\beta(f) \geq \gamma(f)$ for all $f \in K\{T\}$ if and only if $\beta = H(\gamma, \rho)$ for some $\rho \in [0,1]$. The analogous statement for $\calM(B^I_{L,E})$ also holds; for the case $E = \QQ_p$, see \cite[Theorem~7.12]{kedlaya-witt}.
\end{remark}

\section{Structure of rational localizations}
\label{sec:pid}

We next convert our previous observations into some structural properties of the rings obtained from $B^I_{L,E}$ by the formation of rational localizations.

\begin{hypothesis}
Throughout \S\ref{sec:pid}, let $I = [s,r]$ be a closed interval in $(0, +\infty)$.
Let $(B^I_{L,E}, B^{I,+}_{L,E}) \to (C,C^+)$ be a rational localization, i.e., the homomorphism representing a rational subspace of $\Spa(B^I_{L,E}, B^{I,+}_{L,E})$
as in \cite[Lemma 2.4.13]{part1}.
\end{hypothesis}

\begin{convention} \label{conv:extension}
Throughout \S\ref{sec:pid}, we will write $L'$ for an unspecified perfect overfield of $L$ complete with respect to a multiplicative nonarchimedean norm extending the one on $L$.
For such $L'$,
let $(C', C^{\prime +})$ denote the base extension of $(C,C^+)$ along
$(B^I_{L,E}, B^{I,+}_{L,E}) \to (B^I_{L',E}, B^{I,+}_{L',E})$.
\end{convention}

\begin{lemma} \label{L:same localization}
Let $(R,R^+) \to (S,S^+)$ be a rational localization of adic Banach rings
such that $R$ and $S$ are both noetherian. (For instance,
by Theorem~\ref{T:strongly noetherian Robba2} we may take
$R = B^I_{L,E}, S = C$.)
For any $\gothm \in \Maxspec(R)$ 
such that $R/\gothm \cong \calH(\beta)$ for some $\beta \in \calM(S)$
and any positive integer $n$,
the map $R/\gothm^n \to S/\gothm^n S$ is an isomorphism.
\end{lemma}
\begin{proof}
We follow \cite[Proposition~7.2.2/1]{bgr}.
In the commutative diagram
\[
\xymatrix{
R \ar[r] \ar[d] & S \ar[d] \ar@{-->}[ld] \\
R/\gothm^n \ar[r] & S/\gothm^n S
}
\]
the dashed arrow exists and is unique for $n=1$ by hypothesis,
and hence for all $n$ by the universal property of rational localizations.
Since the vertical arrows are surjective, so are $S \to R/\gothm^n$
and $R/\gothm^n \to S/\gothm^n S$. On the other hand, the kernel of $S \to R/\gothm^n$ contains $\gothm^n$ and hence also $\gothm^n S$, so $R/\gothm^n \to S /\gothm^n S$ is also injective.
\end{proof}

\begin{lemma} \label{L:uniform}
The Banach ring $C$ is uniform. In particular, by \cite[Theorem~2.3.10]{part1}, 
the supremum over $\calM(C)$ computes the spectral norm on $C$.
\end{lemma}
\begin{proof}
If $E$ is of characteristic $0$, then
the ring $B^I_{L,E}$ is preperfectoid \cite[Theorem~5.3.9]{part1};
otherwise,  for $E'$ a completed perfect closure of $E$, $B^I_{L,E} \widehat{\otimes}_E E'$ is perfect. In either case, the proof of \cite[Theorem~3.7.4]{part1} implies that $B^I_{L,E}$ is stably uniform.
\end{proof}

\begin{lemma} \label{L:find primitive factor}
Let $x \in C$ be an element which is not a unit. Then for some $L'$, we can find $\overline{u} \in \gothm_{L'} \setminus \{0\}$ such that $\varpi - [\overline{u}]$ divides $x$ in $C'$.
\end{lemma}
\begin{proof}
By Theorem~\ref{T:strongly noetherian Robba2}, the ring $C$ is noetherian, so all of its ideals are closed \cite[Remark~2.2.11]{part1}. Consequently, $C/(x)$ is a nonzero Banach ring, so $\calM(C/(x))$ is nonempty \cite[Theorem~1.2.1]{berkovich1}. 
Choose a point $\gamma \in \calM(C/(x))$ and restrict it to $\beta \in \calM(B^I_{L,E})$,
then take $L', \overline{u}$ as in Lemma~\ref{L:find primitive point}
and put $\beta' = \calH(\varpi - [\overline{u}], 0)$.
By Definition~\ref{D:primitive of degree 1} and Lemma~\ref{L:same localization},
\[
\calH(\beta') \cong B^I_{L',E}/(\varpi - [\overline{u}])B^I_{L',E} \cong C'/(\varpi - [\overline{u}])C',
\]
so $\varpi - [\overline{u}]$ divides $x$ in $C'$.
\end{proof}

\begin{cor} \label{C:primitive factors}
For every $x \in B_{L,E}$ nonzero, we can choose $L'$ such that $x$ factors in $B_{L',E}$
as a unit times a finite product of primitive elements of degree $1$.
(With a more complicated argument, one can force $L'$ to be a completed algebraic closure of $L$; see  \cite[Lemma~6.6(b)]{kedlaya-witt}.)
\end{cor}
\begin{proof}
This follows by Lemma~\ref{L:find primitive factor} and consideration of slopes.
\end{proof}

\begin{lemma} \label{L:convex}
For any $\beta \in \calM(C)$,
there exist finitely many values $0 < \rho_1 < \cdots < \rho_m < 1$
with the following properties.
\begin{enumerate}
\item[(a)]
For $J$ equal to any of $(0, \rho_1], [\rho_1, \rho_2], \dots, [\rho_{m-1}, \rho_m], [\rho_m, 1]$,
either $H(\beta, \rho) \in \calM(C)$ for all $\rho \in J$, or $H(\beta, \rho) \notin \calM(C)$ for all $\rho$ in the interior of $J$.
\item[(b)]
For $J$ as in (a) for which the first alternative holds, for any $x \in C$ the function
\[
t \mapsto \log H(\beta, e^{-t})(x)
\]
is convex and continuous on $-\log J$. In particular, $H(\beta,e^{-t})(x)$ is either zero for all $t \in -\log J$ or nonzero for all $t \in -\log J$.
\end{enumerate}
\end{lemma}
\begin{proof}
There is no harm in enlarging $L$, so we may apply Lemma~\ref{L:find primitive factor} to reduce to the case where $\beta = H(\varpi - [\overline{u}], 0)$
for some $\overline{u} \in \gothm_L \setminus \{0\}$.
As in \cite[Remark~2.4.7]{part1},
we can find $f_1,\dots,f_n \in B_{L,E}$ generating the unit ideal in $B^I_{L,E}$
for which
\[
\Spa(C,C^+) = \{v \in \Spa(B^I_{L,E}, B^{I,+}_{L,E}): v(f_i) \leq v(g) \qquad (i=1,\dots,n)\}.
\]
Apply Corollary~\ref{C:primitive factors} to each of $f_1,\dots,f_n,g$
and let $\varpi - [\overline{u}_1], \dots, \varpi - [\overline{u}_l]$ be the list of all factors obtained.
For each $i \in \{1,\dots,l\}$, note that 
\begin{equation} \label{eq:homotopy on linear factors}
H(\beta,\rho)(\varpi - [\overline{u}_i]) = \max\{p^{-1} \rho, \lambda_r([\overline{u}] - [\overline{u}_i])\};
\end{equation}
from this formula
(and the fact that it suffices to check (b) for $x$ in the dense subring $B_{L,E}[g^{-1}]$ of $C$),
we may easily deduce (a) and (b).
\end{proof}

\begin{cor} \label{C:nonzero by convexity}
Suppose that $C$ is connected.
\begin{enumerate}
\item[(a)]
There exists $\delta \in \calM(C)$ such that $\beta \wedge \delta = \delta$ for all $\beta \in \calM(C)$.
\item[(b)]
For each $\beta \in \calM(C)$, there exists $\rho > 0$ such that
the inverse image of $\calM(C)$ under $H(\beta, \cdot)$ equals $[0,\rho]$ and
$H(\beta,\rho) = \delta$.
\item[(c)]
For any $\beta \in \calM(C)$ of positive radius and any nonzero $x \in C$, we have $\beta(x) \neq 0$.
\end{enumerate}
\end{cor}
\begin{proof}
For $\beta \in \calM(C)$, by Lemma~\ref{L:convex} the inverse image of 
$\calM(C)$ under $H(\beta, \cdot): [0,1] \to \calM(B^I_{L,E})$ 
consists of a finite disjoint union of closed intervals. One of these intervals contains $0$; let $\rho_C(\beta)$ be its right endpoint and put $\delta_C(\beta) = H(\beta,\rho_C(\beta))$.
In this way we define a map $\delta_C: \calM(C) \to \calM(C)$.

We claim that for $\beta, \gamma \in \calM(C)$, $\delta_C(\beta) \wedge \gamma = \delta_C(\beta)$. To check this, put $\rho = \rho_C(\beta)$; we may assume that $\rho < 1$, as otherwise the claim is obvious. For $\epsilon > 0$ sufficiently small, we have $\rho + \epsilon \leq 1$ and $H(\beta, \rho+\epsilon) \notin \calM(C)$.
For each such $\epsilon$, we may form an open partition of $\calM(C)$ by 
intersecting with the connected components of $\calM(B^I_{L,E}) \setminus H(\beta, \rho+\epsilon)$ described in Lemma~\ref{L:complementary components}. Since $\calM(C)$ is connected, it must be contained in the component of $\calM(B^I_{L,E}) \setminus H(\beta, \rho+\epsilon)$ containing $\beta$; in particular,
$H(\beta, \rho+\epsilon) \wedge \gamma = H(\beta, \rho+\epsilon)$ for all $\gamma \in \calM(C)$. Since this holds for $\epsilon > 0$ sufficiently small, it holds also for $\epsilon = 0$, proving the claim.

In particular, for $\beta, \gamma \in \calM(C)$, we have $\delta_C(\beta) \wedge \delta_C(\gamma) = \delta_C(\beta)$. By symmetry, we deduce that the image of the map $\delta_C$ consists of a single point $\delta$. This immediately implies (a) and (b).
(In the context of Remark~\ref{R:geometric analogy}, these statements correspond to the usual description of a connected affinoid subspace of $\calM(K\{T\})$, at least when $K$ is algebraically closed, as a closed disc minus a finite union of open subdiscs.)

Suppose now that $x \in C$ and there exists $\beta \in \calM(C)$ of positive radius with $\beta(x) = 0$.
By Lemma~\ref{L:convex}(b), we deduce that $H(\beta, \rho)(x) = 0$ for all
$\rho \in (0, \rho_C(\beta)]$, and in particular $\delta(x) = 0$.
For any $\gamma \in \calM(C)$, we may apply Lemma~\ref{L:convex}(b) again to deduce that
$H(\gamma,\rho)(x) = 0$ for all $\rho \in (0,\rho_C(\gamma)]$, and hence $\gamma(x) = 0$ by continuity. By Lemma~\ref{L:uniform}, this implies $x=0$.
\end{proof}
\begin{cor} \label{C:domain}
The ring $C$ is a finite direct sum of integral domains.
\end{cor}
\begin{proof}
Since $C$ is noetherian by Theorem~\ref{T:strongly noetherian Robba2}, it is a finite direct sum of connected subrings, each of which is a domain by Corollary~\ref{C:nonzero by convexity} and the fact that $\calM(C)$ contains a point of nonzero radius.
\end{proof}

\begin{lemma} \label{L:finitely many zeroes}
Suppose that $C$ is connected.
For $x \in C$ nonzero, there are only finitely $\beta \in \calM(C)$ for which $\beta(x) = 0$.
\end{lemma}
\begin{proof}
By Corollary~\ref{C:domain}, $C$ is an integral domain.
It suffices to check that for each $\beta \in \calM(C)$ for which $\beta(x) = 0$, there exists a neighborhood $U$ of $\beta$ in $\calM(C)$ such that $\gamma(x) \neq 0$ for $\gamma \in U \setminus \{\beta\}$. Note that $\beta$ must have radius 0 by
Corollary~\ref{C:nonzero by convexity}.

By Lemma~\ref{L:find primitive point}, we can choose some $L'$ and some $\overline{u} \in \gothm_{L'} \setminus \{0\}$ such that $\beta' = H(\varpi-[\overline{u}], 0)$ restricts to $\beta$.
By Corollary~\ref{C:domain}, $C'$ is a finite direct sum of integral domains.
Let $\gothm$ be the ideal of $B^I_{L',E}$ generated by $\varpi - [\overline{u}]$;
then $B^I_{L',E}/\gothm \cong \calH(\beta')$, so in particular $\gothm$ is maximal.

Suppose by way of contradiction that $x \in \gothm^n C'$ for all $n$.
By Krull's intersection theorem \cite[Corollary~5.4]{eisenbud},
$x$ then vanishes in the local ring $C'_{\gothm}$, and hence in the connected
component of $C'$ whose spectrum contains $\gothm$.
In particular, there exists an open neighborhood $U'$ of $\beta'$ in $\calM(C')$ such that $\gamma(x) = 0$ for all $\gamma \in U'$; since $\beta$ has radius 0, $U'$ restricts to a neighborhood $U$ of $\beta$ in $\calM(C)$. However, any such $U$ contains points of positive radius, contradicting Corollary~\ref{C:nonzero by convexity}.

By Lemma~\ref{L:same localization}, we can find some $n$ such that $(\varpi - [\overline{u}])^n$ divides $x$ in $C'$
and the quotient $y$ has nonzero image in $\calH(\beta')$.
We may thus choose a neighborhood $U'$ of $\beta'$ in $\calM(C')$ such that $\gamma(y) \neq 0$ for all $\gamma \in U'$, and hence $\gamma(x) \neq 0$ for all $\gamma \in U' \setminus \{\beta'\}$. Since $\beta$ has radius 0, $U'$ restricts to a neighborhood of $U$ in $\calM(C)$ of the desired form.
\end{proof}

\begin{theorem} \label{T:strong PID}
The ring $C$ has the following properties.
\begin{enumerate}
\item[(a)]
The ring $C$ is a direct sum of finitely many noetherian integral domains $C_1,\dots, C_n$.
\item[(b)]
For $i=1,\dots,n$, every element of $C_i$ can be written as an element of $W(\gotho_L)$ times a unit.
\item[(c)]
For $i=1,\dots,n$, $C_i$ is a principal ideal domain.
\end{enumerate}
\end{theorem}
The fact that $B^I_{L,E}$ itself is a principal ideal domain was known previously;
see \cite[Proposition~2.6.8]{part1}.
\begin{proof}
We have (a) thanks to Theorem~\ref{T:strongly noetherian Robba2} and 
Corollary~\ref{C:domain}. We may thus assume hereafter that $C$ itself is a noetherian integral domain.

Choose any nonzero $x \in C$.
By Lemma~\ref{L:finitely many zeroes},
there are only finitely many $\beta \in \calM(C)$ for which $\beta(x) = 0$.
If there are no such $\beta$, then $x$ is a unit by \cite[Corollary~2.3.5]{part1}.
Otherwise, by Lemma~\ref{L:find primitive factor}, each such $\beta$ may be lifted to
$H(\varpi - [\overline{u}], 0)$ for some $L'$ and some $\overline{u} \in \gothm_{L'} \setminus \{0\}$. We may make a single choice of $L'$ and then let $\overline{u}_1,\dots,\overline{u}_l$ be the resulting values of $\overline{u}$.
We may then apply Lemma~\ref{L:primitive descent} to the product
$\prod_{i=1}^l (\varpi - [\overline{u}_i])$ to write it as a unit in $W(\gotho_{L'})$ times some element $y_0 \in W(\gotho_L)$, which then must be a divisor of $x$ in $C$. 
We thus form a sequence $x_0 = x, x_1, \dots$ of elements of $C$ in which for each $i \geq 0$, we have $x_i = y_i x_{i+1}$ for some $y_i \in W(\gotho_L)$ which is not a unit in $C$. Since $C$ is noetherian, we cannot extend this sequence indefinitely; we then have that $x$ is the product of the $y_i$ times a unit. This proves (b), which implies (c) because
$A^r_{L,E}$ is a principal ideal domain by Corollary~\ref{C:pid1}.
\end{proof}

\section{Structure of \'etale morphisms}
\label{sec:etale}

To conclude, we extend the preceding results to \'etale morphisms.
\begin{hypothesis}
Throughout \S\ref{sec:etale},
let $(B^I_{L,E}, B^{I,+}_{L,E}) \to (C,C^+)$ be a morphism of adic Banach rings which is \'etale in the sense of Huber \cite[Definition~1.6.5]{huber}. In particular, $C$ is a quotient of $B^I_{L,E}\{T_1,\dots,T_n\}$ for some $n$, so it is again strongly noetherian by Theorem~\ref{T:strongly noetherian Robba2}.
\end{hypothesis}

\begin{lemma} \label{L:etale factorization}
There exist finitely many rational localizations $\{(C,C^+) \to (D_i, D_i^+)\}_i$
such that $\cup_i \Spa(D_i, D_i^+) = \Spa(C,C^+)$ and for each $i$,
$(B^I_{L,E}, B^{I,+}_{L,E}) \to (D_i,D_i^+)$ factors as a connected rational localization $(B^I_{L,E}, B^{I,+}_{L,E}) \to (C_i,C_i^+)$
followed by a finite \'etale morphism $(C_i, C_i^+) \to (D_i, D_i^+)$ with $D_i$ also connected.
\end{lemma}
\begin{proof}
See \cite[Lemma~2.2.8]{huber}.
\end{proof}

\begin{remark} \label{R:etale factorization}
We record some immediate consequences of Lemma~\ref{L:etale factorization}.
\begin{enumerate}
\item[(a)]
By \cite[Theorem~2.2]{huber2}, the structure presheaf on $\Spa(C,C^+)$ is a sheaf,
so the map $C \to \bigoplus_i D_i$ is injective.
\item[(b)]
By Theorem~\ref{T:strong PID}(c),
the ring $D_i$ is a connected finite \'etale algebra over the principal ideal domain $C_i$.
It is therefore a Dedekind domain.
\item[(c)]
By Lemma~\ref{L:uniform} and \cite[Proposition~2.8.16]{part1}, $D_i$ is uniform.
By (a), this implies that $C$ is uniform.
\end{enumerate}
\end{remark}

\begin{lemma} \label{L:domain2}
Suppose that $C$ is connected. Then $C$ is an integral domain; moreover, for any rational localization $(C,C^+) \to (D,D^+)$, the map $C \to D$ is injective.
\end{lemma}
\begin{proof}
Set notation as in Lemma~\ref{L:etale factorization}.
By Remark~\ref{R:etale factorization}(b), each $D_i$ is a domain; in particular, the norm map from $D_i$ to $C_i$ takes nonzero elements to nonzero elements. By Corollary~\ref{C:nonzero by convexity}, if $x \in C$ has nonzero image in $D_i$, then $\beta(x) \neq 0$ for any $\beta \in \calM(D_i)$ restricting to a point of $\calM(C_i)$ of positive radius.

Suppose that $x \in C$ maps to zero in $D_i$ for some $i$. For any $j$, if
$\calM(D_i) \cap \calM(D_j)$ is nonempty, then it contains a point restricting to a point of $\calM(C_i) \cap \calM(C_j)$ of positive radius. By the previous paragraph, this implies that $x$ also maps to zero in $D_j$. Since $C$ is connected, it follows that $x$ maps to zero in $D_i$ for all $i$; by Remark~\ref{R:etale factorization}(a), this means $x=0$. Consequently, the maps $C \to D_i$ are injective; since each $D_i$ is a domain, so then is $C$.

Let $(C,C^+) \to (D,D^+)$ be a general rational localization. Augment the collection of rational localizations described in Lemma~\ref{L:etale factorization} by the compositions of $(C,C^+) \to (D,D^+)$ with a similar collection of rational localizations of $D$.
Then the previous conclusions still apply, but now there exists an index $i$ such that $C \to D_i$ factors through $D$. It follows that $C \to D$ is itself injective.
\end{proof}

\begin{lemma} \label{L:finitely many zeroes2}
Suppose that $C$ is connected. For $x \in C$ nonzero, there are only finitely many $\beta \in \calM(C)$ for which $\beta(x) \neq 0$.
\end{lemma}
\begin{proof}
Set notation as in Lemma~\ref{L:etale factorization}.
By Lemma~\ref{L:domain2}, $x$ has nonzero image in each $D_i$.
Let $y_i$ be the norm of $x$ from $D_i$ to $C_i$.
By Lemma~\ref{L:finitely many zeroes}, there are only finitely many $\beta \in \calM(C_i)$ such that $\beta(y_i) \neq 0$. Since the map $\calM(D_i) \to \calM(C_i)$ has finite fibers,
there are only finitely many $\beta \in \calM(D_i)$ such that $\beta(x) \neq 0$. This proves the claim.
\end{proof}

\begin{cor} \label{C:maximal ideal one point}
For $\gothm \in \Maxspec(C)$, $C/\gothm \cong \calH(\beta)$ for some $\beta \in \calM(C)$.
(We may thus view $\Maxspec(C)$ as a subset of $\calM(C)$.)
\end{cor}
\begin{proof}
By Lemma~\ref{L:finitely many zeroes2}, $\calM(C/\gothm)$ is a finite compact topological space; however, by
\cite[Proposition~2.6.4]{part1}, $\calM(C/\gothm)$ is also connected. It is thus a singleton set $\{\beta\}$ for some $\beta \in \calM(C)$.
By \cite[Theorem~2.3.10]{part1}, the norm on $C/\gothm$ is equivalent to $\beta$; since $C/\gothm$ is also complete, we have $C/\gothm \cong  \calH(\beta)$.
\end{proof}

\begin{remark}
If $A$ is an arbitrary commutative Banach ring whose underlying ring is a field, 
the space $\calM(A)$ need not be reduced to a point; for instance, consider $A = \QQ$ equipped with the supremum of the trivial norm and the $p$-adic norm for some prime $p$.
However, we do not know of such an example containing a topologically nilpotent unit.
\end{remark}

\begin{theorem} \label{T:Dedekind}
The ring $C$ is a direct sum of finitely many Dedekind domains.
\end{theorem}
\begin{proof}
Since $C$ is noetherian by Theorem~\ref{T:strongly noetherian Robba2},
we may reduce to the case where $C$ is connected, and hence an integral domain by Lemma~\ref{L:domain2}.
It thus remains to prove that for any maximal ideal $\gothm$ of $C$, the local ring
$C_\gothm$ is principal. Since $C$ is noetherian, we may check this after completion;
by Corollary~\ref{C:maximal ideal one point} and Lemma~\ref{L:same localization}, this completion remains unchanged after replacing $C$ with $D$ for any rational localization $(C,C^+) \to (D,D^+)$ such that $\gothm D \neq D$.
We may thus deduce the claim from Lemma~\ref{L:etale factorization}
and Remark~\ref{R:etale factorization}(b).
\end{proof}

\begin{remark}
From Lemma~\ref{L:domain2}, Theorem~\ref{T:Dedekind}, and the fact that any torsion-free module over a Dedekind domain is flat, it follows that for any rational localization $(C,C^+) \to (D,D^+)$, the map $C \to D$ is flat.
On the other hand, since $C$ is strongly noetherian, we may apply \cite[Lemma~1.7.6]{huber} to deduce the same conclusion even when $(C,C^+) \to (D,D^+)$ is \'etale.
\end{remark}

\begin{remark} \label{R:additional properties}
In general, the properties of $C$ are analogous to the properties of one-dimensional affinoid algebras over a field (compare Remark~\ref{R:geometric analogy}). By this analogy, we expect the following additional properties to hold.
\begin{itemize}
\item
The points of $\Spa(C,C^+)$ can be classified into types 1--5 by analogy with the points of an analytic curve over a field, with the points of types 1--4 appearing in $\calM(C)$,
and points of $\Maxspec(C)$ giving rise to points of $\calM(C)$ of type 1.
(Compare \cite[Theorem~8.17]{kedlaya-witt}.)
\item
If $C^+$ equals $C^\circ$ (the ring of power-bounded elements of $C$), then for any rational localization $(C,C^+) \to (D,D^+)$, one also has $D^+ = D^\circ$. (Compare \cite[Lemma~2.5.9(d)]{part1}.)
\item
A finite collection of rational subspaces of $\Spa(C,C^+)$ forms a covering 
if and only if the intersections with $\Maxspec(C)$ do so.
Consequently, for $(C,C^+) \to (D,D^+)$ a rational localization corresponding to the subspace $U$ of $\Spa(C,C^+)$, $D$ depends only on $U \cap \calM(C)$.
(Compare \cite[Lemma~2.5.12, Corollary~2.5.13]{part1}.)
\item
Relative Nullstellensatz: for any $n \geq 0$ and any
$\gothm \in \Maxspec(C\{T_1,\dots,T_n\})$, $\gothm \cap C \in \Maxspec(C)$.
(This fails if $C\{T_1,\dots,T_n\}$ is replaced by $C\{T_1/\rho_1,\dots,T_n/\rho_n\}$.)
\item
For any nonnegative integer $n$ and any $\rho_1,\dots,\rho_n>0$, the ring
$C\{T_1/\rho_1,\dots,T_n/\rho_n\}$ is regular and excellent. (Compare Remark~\ref{R:regular excellent}.)
\end{itemize}
\end{remark}

\end{document}